\documentclass{amsart}

\usepackage[utf8]{inputenc}
\usepackage{graphicx}
\usepackage{amsmath}
\usepackage{amsfonts}
\usepackage{amsthm}
\usepackage{enumerate}
\usepackage{amsrefs}
\usepackage{hyperref}

\newcommand{\R}{\mathbb{R}}
\newcommand{\N}{\mathbb{N}}

\newcommand{\B}{\mathcal{B}}

\newtheorem{theorem}{Theorem}
\newtheorem{lemma}[theorem]{Lemma}
\newtheorem{remark}[theorem]{Remark}
\newtheorem{proposition}[theorem]{Proposition}

\newtheorem{corollary}[theorem]{Corollary}

\newtheorem{conjecture}[theorem]{Conjecture}

\author{Artur B. Saturnino}
\address{Department of Mathematics, University of Pennsylvania,
	Philadelphia, PA 19104, USA}
\email{bsatur@sas.upenn.edu}
\title[On the genus and area of CMC surfaces with bounded index]
{On the genus and area of constant mean curvature surfaces with bounded index}

\begin{document}
\maketitle

\begin{abstract}
Using the local picture of the degeneration of sequences of minimal surfaces
developed by Chodosh, Ketover and Maximo
\cite{Davi-Minimal-hypersurfaces-with-bdd-index} we show that in any closed
Riemannian 3-manifold $(M,g)$, the genus of an embedded CMC surface can be
bounded only in terms of its index and area, independently of the value of its
mean curvature. We also show that if $M$ has finite fundamental group, the genus
and area of any non-minimal embedded CMC surface can be bounded in term of 
its index and a lower bound for its mean curvature.
\end{abstract}

\section{Introduction}\label{sec: intro} 

Let $\Sigma$ be a compact immersed two-sided constant mean curvature (CMC)
surface in a Riemannian 3-manifold $(M,g)$. The \textit{strong Morse index}
of $\Sigma$ is defined as the index of the stability operator
\[Q(u,u) = \int_\Sigma |\nabla u|^2 - (|A|^2 + \text{Ric}_M(N))u^2 dvol_\Sigma\]
where $|A|$ is the norm of the second fundamental form of $\Sigma$
and $u \in C^{\infty}(\Sigma)$
\footnote{In the case where the surface is not minimal, it
is natural to consider only variations of $\Sigma$ that preserve volume to first
degree. That means adding the restriction $\int_\Sigma u = 0$ to the domain of
$Q$. The index of $Q$ in this domain is called the \textit{weak Morse index} of
$\Sigma$. Note that the difference between the weak and strong Morse indices
is at most one \cite{Barobsa-Berard-Eigenvalues-and-twisted}.
}.
We will refer to the strong Morse index of $\Sigma$ as the index of $\Sigma$.

The index is a natural variational quantity associated to CMC surfaces, hence it
is expected to be controlled in CMC surfaces produced through variational methods.
For example, it has been shown that the index
of minimal surfaces associated to the volume spectrum are controlled by the
order of the associated spectrum element \cites{zhou2019multiplicity,
li2020improved}, and analogous results are believed to hold for CMC surfaces 
produced trough the
Zhou-Zhu min-max procedure \cites{Zhou-Zhu-Min-max-theory-for-CMC,li2020improved}.

On the other hand, the geometry and topology of surfaces produced through
variational methods tend to be hard to control directly. With the objective of 
bridging this gap, we are interested in the relation between the index of CMC 
surfaces and classical geometric and topological quantities. This relation
has been well explored for minimal surfaces and hypersurfaces 
\cites{
Ros-One-Sided-comple-stable, Davi-On-the-topology-and-index-of-min-surf,
Song-Morse-index-betti-numbers-and-singular-set,
Savo-Index-bounds-for-minimal-hypersurfaces-of-the-sphere,
Ambrozio-Carlotto-Sharp-Comparing-the-Morse-index,
Davi-Minimal-hypersurfaces-with-bdd-index
},
although the following conjecture
from Marques and Neves' 2014 ICM lectures is still open:
\begin{conjecture}[{\cite{Neves-New-applications-of-min-max}}]
If the ambient manifold has
positive Ricci curvature, then an index $I$ embedded orientable compact minimal
hypersurface has first Betti number bounded by a fixed multiple of $I$.
\end{conjecture}

The relation between the index and the geometry of CMC surfaces and
hypersurfaces is well known in some special cases (e.g.
\cites{
Cavalcante-Lower-bound-for-the-index,
Rossman-Morse-index-of-cmc-tori,
Hong-index-estimates-for-cmc
}).
Here we will study the area and genus of CMC surfaces of bounded index by
describing the degeneration of sequences of such surfaces
in a similar way as Chodosh, Ketover and Maximo have done for
minimal surfaces \cite{Davi-Minimal-hypersurfaces-with-bdd-index}. We will show
that in any closed 3-manifold, the genus of a CMC surface is controlled by
its index and area independently of the value of its mean curvature:

\begin{theorem}\label{thm: topology bound}
Let $I, A_0 < \infty$ and suppose $(M,g)$ is a closed 3-dimensional Riemannian
manifold. There is a constant $C < \infty$ depending only on $(M,g), I$
and $A_0$ such that any closed, connected, embedded CMC surface in $(M,g)$ with area
at most $A_0$ and index at most $I$ has genus at most $C$.
\end{theorem}

If $M$ is spherical, that is, $M$ is closed with finite fundamental group,
we will show that every sequence of closed connected CMC surfaces embedded in
$M$ with an uniform lower mean curvature bound and an uniform upper index bound
has a subsequence that does not accumulate away from a finite set or points.
Using this fact we are able to show:

\begin{theorem}\label{thm: area bound}
Let $I < \infty, \eta > 0$ and suppose $(M,g)$ is a spherical 3-dimensional Riemannian
manifold. There are constants $B, C < \infty$ depending only on $(M,g), I$ and 
$\eta$ such that any closed connected embedded CMC surface in $(M,g)$
with index at most $I$ and mean curvature at least $\eta$ has area at most $B$
and genus at most $C$.
\end{theorem}

The compactness assumption in Theorem \ref{thm: area bound} is necessary
\cite{Baris-H-surfaces-with-arbitrary-topology}, as is the lower bound on the
mean curvature \cite{Colding-Minicozzi-Examples-of-embeded-min-tori}. However it is 
not clear if the assumption on the fundamental group of $M$ is essential.

We will always assume the surfaces are connected and two-sided. Note that if
$\Sigma$ is a non-minimal CMC surface we can orient its normal bundle
by the mean curvature vector. Then the scalar value of the mean curvature
$H = \text{tr}(A)$ of $\Sigma$ is positive, we will often refer to $H$ as the 
mean curvature of $\Sigma$. In the case $\Sigma$ is minimal and one-sided,
the proof of Theorem \ref{thm: topology bound} is given in
\cite{Davi-Minimal-hypersurfaces-with-bdd-index}. Now we will now present the main ideas
in the proof of theorems \ref{thm: topology bound} and \ref{thm: area bound}.

\subsection{Outline of the proofs of the main theorems}
Theorems \ref{thm: topology bound} and \ref{thm: area bound} are shown by studying 
the degeneration of sequences of closed
embedded CMC surfaces, possibly with varying mean curvature, $\{\Sigma_n\}$ 
in $(M,g)$ with index at most $I$. Assume by contradiction that the genus of
the $\Sigma_n$ form a divergent sequence, and assume for simplicity that, passing
to a subsequence, the $\Sigma_n$ have uniformly bounded mean curvature (if this
is not the case an extra rescaling is needed).

By Theorem \ref{thm: curvature bounds}, we can pass $\{\Sigma_n\}$ to a
subsequence with uniformly bounded second fundamental form away from a
set of at most $I$ points in $M$, called blow-up points. Near the blow-up 
points we follow the arguments of Chodosh, Ketover and Maximo 
\cite{Davi-Minimal-hypersurfaces-with-bdd-index} to show that a subsequence of 
rescaled $\Sigma_n$ must converge to a minimal surface in $\R^3$ with index
at most $I$. This fact allows us to bound the genus and area of the $\Sigma_n$ 
near the blow-up points.

Away from the blow-up points the $\Sigma_n$ must subconverge to a weak CMC 
lamination of $M$
\footnote{
Even though we will not use this fact, it is interesting to note that by Theorem
\ref{thm: curvature bounds} and the Local Removable Singularity Theorem for CMC
laminations \cite{Meeks-Perez-Ros-Classificaion-of-CMC-foliations}*{Theorem 1.2}
we can choose a subsequence so this limit lamination extends to the blow-up
points.
}.
If the leaves of this lamination are proper we can use the limit lamination to
bound the genus of the $\Sigma_n$ away from the blow-up points. In order to show
that the leaves of the limit lamination are proper we bound the number of sheets
any $\Sigma_n$ can have in any small region away from the blow-up points. At this 
point the proof of theorems \ref{thm: topology bound} and \ref{thm: area bound} diverge.

In the case of Theorem \ref{thm: topology bound} we have uniform area bound for
the $\Sigma_n$, so the bound of the number of sheets follows from a
standard argument. To prove Theorem \ref{thm: area bound} we will
show that if $M$ is spherical, the direction of the mean curvature vectors of
these sheets must alternate in a certain sense. This, together with an uniform lower 
bound on the mean curvature of the $\Sigma_n$ will give uniform bounds on the 
maximum number of sheets of any $\Sigma_n$ that can pass through a small region away
from the blow-up points.

\subsection{Outline of the paper}
In Section \ref{sec: curv bounds} we will use blow-up arguments to
show curvature bounds for CMC surfaces with bounded index. Section 
\ref{sec: local picture near blow-up pts} is dedicated to adapting results from 
Chodosh, Ketover and Maximo \cite{Davi-Minimal-hypersurfaces-with-bdd-index} to
CMC surfaces. In Section \ref{sec: picture away from blow-up} we show upper bounds
on the number of sheets of a CMC surface that pass through a small region of a 
spherical 3-manifold. Finally, Section \ref{sec: pf of main theorems} contains
the proof of the main theorems.
Appendices \ref{app: uniform graph} and \ref{app: alternating CMC graphs}
contain details about the local parametrization of CMC surfaces as graphs.

\subsection{Acknowledgments} 
I would like to thank my advisor Davi Maximo for 
many insightful comments and conversations. I would also like to
thank the reviewer for carefully reading the manuscript, pointing
out typos and making observations that led to corrections to
Proposition \ref{prop: alternating surfaces} and Lemma 
\ref{lemma: local picture away from blow-up}.

\section{Curvature bounds for CMC surfaces with finite index}
\label{sec: curv bounds}

The following scale-invariant inequality is a generalization of 
a classical inequality first discovered by Schoen 
\cite{Schoen-Estimates-for-stable-min-surf} (c.f. 
\citelist{
\cite{Davi-Minimal-hypersurfaces-with-bdd-index}*{Lemma 2.2}
\cite{Rosenberg-General-Curvature-Estimates}*{Main Theorem}
}).
We will prove the theorem in full generality since we believe it may be
useful in more contexts.
\begin{theorem}\label{thm: curvature bounds}
For all $I < \infty$ there is a constant $C < \infty$ depending only on $I$
such that the following holds:
Let $(M,g)$ be a complete Riemannian 3-manifold with absolute 
sectional curvature bounds $|K| \le \Lambda < \infty$ for some $\Lambda > 0$ and
suppose $\Sigma \to M$ is an immersed CMC surface with (strong)
index at most $I$. Then there is a set $\B \subset \Sigma$ with at most
$I$ points such that for all $p \in \Sigma$
\begin{equation}\label{eq: general curvature bound}
	|A|(p)\min\left\{d_\Sigma(p, \partial \Sigma \cup \B),
	(\sqrt{\Lambda})^{-1}\right\} \le C. 
\end{equation}
\end{theorem}

This Theorem will be shown in \ref{subsec: pf of curv bound} after we introduce
the language we will use to prove it. Note that since 
the constant $C$ does not depend on the mean curvature $H$ of
$\Sigma$ and $|A| \ge \sqrt 2 H$ we can conclude from \eqref{eq: general curvature
bound} that if $\frac{\sqrt 2H}{\sqrt \Lambda}\ge C$
we must have $d_\Sigma(p, \partial \Sigma \cup \B) \le \frac{C}{\sqrt 2H}$ for all $p \in
\Sigma$. So it follows:

\begin{corollary}\label{cor: diam bound}
Take $M, \Lambda, \Sigma, I$ an $C$ as in Theorem \ref{thm: curvature bounds}
and let $H$ be the mean curvature of $\Sigma$. If 
$\frac{\sqrt 2 H}{\sqrt \Lambda} \ge C$ we can conclude that for all $p \in
\Sigma$:
\begin{equation}\label{eq: diameter bound boundary case}
	d_\Sigma(p, \partial \Sigma) \le \left(I + \frac{1}{2}\right)\frac{\sqrt 2 C}{H}
\end{equation}
if $\Sigma$ has a boundary, and
\begin{equation}\label{eq: diameter bound closed case}
	\textup{diam}_\Sigma(\Sigma) \le \frac{\sqrt 2 IC}{H}
\end{equation}
if $\Sigma$ is closed.
\end{corollary}
\begin{proof}
We will only show \eqref{eq: diameter bound boundary case} as the proof of
\eqref{eq: diameter bound closed case} is very similar.
As observed above,
\begin{equation}\label{eq: d_Sgima p}
d_\Sigma(p, \partial \Sigma \cup \B) \le \frac{C}{\sqrt 2H}
\end{equation}
for all $p \in \Sigma$. So the distance between any point in $\B$
and its complement in $\partial \Sigma \cup \B$ is at most 
$\frac{\sqrt 2 C}{H}$.
Then an induction argument on $I$ shows that any point
in $\B$ is at distance at most $I\frac{\sqrt 2 C}{H}$ from $\partial \Sigma$.
Now \eqref{eq: diameter bound boundary case} follows from this
fact together with \eqref{eq: d_Sgima p}.
\end{proof}

\subsection{Harmonic coordinates and sequences of rescaled surfaces}
\label{subsec: harm coords and sequences of}

Theorem \ref{thm: curvature bounds} does not depend on the
manifold $M$, only on its sectional curvature bounds. For this reason we will
have to make use of harmonic coordinates. More specifically we will use the
following result by Rosenberg, Souam and Toubiana:

\begin{theorem}[{\cite{Rosenberg-General-Curvature-Estimates}*{Theorem 2.1}}]
\label{thm: harmonic coords}
Let $\alpha \in (0,1)$ and $\delta >0$. Let $(M,g)$ be a
Riemannian 3-manifold without boundary with absolute sectional curvature bounds
$|K| \le \Lambda < \infty$. Let $\Omega$ be a open subset of $M$ and set
\[\Omega(\delta) = \{x \in M : d_M(x, \Omega) < \delta \}\]
where $d_M$ is the distance associated to $g$. Suppose there exists $i >0$ such
that for all $x \in \Omega(\delta)$ have $\text{inj}_{(M,g)}(x) \ge i$, where 
$\text{inj}_{(M,g)}(x)$ is the injectivety radius of $M$ at $x$. Then there exists
constants $Q_0, r_0 > 0$ depending only on $i, \delta,
\Lambda$ and $\alpha$, and not on $M$, such that for any $x \in \Omega$, there
exists a harmonic coordinate chart $(U, \varphi, B_M(x, r_0))$, where $U
\subset \R^3$ is an open set containing the origin, $B_M(x,r_0)$ is the
geodesic ball in $M$ centered at $x$ with radius $r_0$, and $\varphi: U \to
B_M(x, r_0)$ is such that $\varphi(0) = x$, and such that the
metric tensor $\varphi^*g$ is $C^{1,\alpha}$- controlled. Namely the 
components $(\varphi^*g)_{ij}$, $i,j = 1,2,3$ of $\varphi^*g$ satisfy:
\begin{equation}\label{eq: quadatic form bound}
	Q_0^{-1}\delta_{ij} \le (\varphi^*g)_{ij} \le Q_0\delta_{ij}
\end{equation}
as quadratic forms, and
\begin{equation}\label{eq: holder bound on tensor}
	\|(\varphi^*g)_{ij}\|_{C^{1,\alpha}(U)} \le Q_0.
\end{equation}
\end{theorem}

The proofs of all theorems presented in this paper make use of sequences of 
rescaled surfaces. What we mean by rescaled surfaces will be
slightly different for different proofs. Following
\cite{Rosenberg-General-Curvature-Estimates}, we will present here the
definition used in the proof of Theorem \ref{thm: curvature bounds}, which is
more involved because the surfaces might only be immersed and
because the ambient space is not fixed.

Suppose $\{(M_n,g_n)\}_{n \in \N}$ is a sequence of complete 3-manifolds with
uniformly bounded absolute sectional curvature $ |K_{n}| \le \Lambda < \infty$.
For each $n \in \N$ let $\Sigma_n$ be an immersed $H_n$-surface in $M_n$ (that
is, a constant mean curvature surface with mean curvature $H_n$), possibly with
boundary, and take a sequence of points $p_n \in \Sigma_n \setminus
\partial\Sigma_n$. For any sequence of positive numbers $\{\sigma_n\}_{n \in \N}$ 
with $\sigma_n \to \infty$ We will define a \textit{sequence of rescaled surfaces} 
\[\sigma_n(\Sigma_n - p_n)\]
immersed in $\R^3$.

By abuse of notation, denote image of $p_n$ under the immersion $\Sigma_n \to M_n$ by $p_n$. 
Since the exponential map of any of the $M_n$ at any point is a local differomorphism when
restricted to the euclidean ball of radius $\frac{\pi}{\sqrt \Lambda}$ we may
endow this ball with the pull-back metric $\exp_{p_n}^* g_n$, which we will also
denote by $g_n$.
By \cite{Rosenberg-General-Curvature-Estimates}*{Lemma 2.2} the closed ball
$\overline{B_{\R^3}(0, \frac{\pi}{2\sqrt \Lambda})} \subset (B_{\R^3}(0,
\frac{\pi}{\sqrt \Lambda}), g_n)$
has injectivity radius at least $ \frac{\pi}{4\sqrt \Lambda}$. Fix an $\alpha \in (0,1)$ and
consider $r_0, Q_0$ given by applying Theorem \ref{thm: harmonic coords}
with $\delta = \frac{\pi}{8\sqrt \Lambda}$ and $i = \frac{\pi}{4\sqrt \Lambda}$.
For all $n \in \N$ take a
harmonic chart $(U_n, \varphi_n, B_{\R^3}(0, r_0))$ with $0 \in U_n$ and 
$\varphi_n(0) = 0$. We define 
\[\Sigma_n - p_n \subset (U_n, \varphi_n^*g_n)\] 
to be the connected component of
$(\exp_{p_n} \circ \varphi_n)^{-1}(\Sigma_n)$ passing trough the origin. Note
that this is an immersed $H_n$-surface. For all $\sigma>0$ let $\mu_\sigma: \R^3 \to \R^3$ be
the dilation $x \mapsto \sigma x$. Define $\tilde g_n = (\mu_{\sigma_n} \circ \varphi_n)^*g_n$
It is clear that
\[\sigma_n(\Sigma_n - p_n) \subset (\mu_{\sigma_n}(U_n), \tilde g_n)\]
is an immersed $\frac{H_n}{\sigma_n}$-surface. 
We will often use that $(\mu_{\sigma_n}(U_n), \tilde
g_n) \to (\R^3, \delta_{ij})$ locally in $C^{1,\alpha}$ 
\cite{Rosenberg-General-Curvature-Estimates}.

\subsection{Proof of Theorem \ref{thm: curvature bounds}}\label{subsec: pf of
curv bound}

Following the arguments of Chodosh, Ketover and Maximo
\cite{Davi-Minimal-hypersurfaces-with-bdd-index}*{Lemma 2.2} we
will use induction on $I$. The case $I = 0$ is done in
\cite{Rosenberg-General-Curvature-Estimates}. Let $I >0$ and assume by
contradiction that the theorem does not hold for $I$. 
Then there is a sequence $\{(M_n,g_n)\}_{n \in \N}$ of complete $3$-manifolds
with absolute sectional curvature  bounds $|K_{n}| \le \Lambda_n< \infty$
and such that for each $n \in \N$ there is an
$H_n \ge 0$ and an immersed $H_n$-surfaces $\Sigma_n \to M_n$ and 
points $q_n \in \Sigma_n$ such that
\[|A_n|(q_n)\min\{d_{\Sigma_n}(q_n, \partial \Sigma_n), (\sqrt \Lambda_n)^{-1} \} > n\]
where $|A_n|$ is the norm of the second fundamental form of $\Sigma_n$.
Since the left hand side of the inequality above is invariant under rescaling 
$g_n$ we can assume that all $\Lambda_n$ are equal to a constant $\Lambda$.

Let $R_n = \min\{d_{\Sigma_n}(q_n, \partial \Sigma_n), (\sqrt \Lambda)^{-1}\}$ and
denote $B_{\Sigma_n}(q_n, R_n)$ by $D_n$. Choose $p_n \in \Sigma_n$ so that
\[|A_n|(p_n)d_{D_n}(p_n,\partial D_n) = 
\max_{p \in D_n} |A_n|(p)d_{D_n}(p,\partial D_n).\]
Let $\tilde R_n = d_{D_n}(p_n,\partial D_n)$ and denote $B_{D_n}(p_n, \tilde R_n)$ by $\tilde
D_n$.
By construction
\[|A_n|(p_n) d_{\tilde D_n}(p_n,\partial \tilde D_n) = 
\max_{p \in \tilde D_n} |A_n|(p)d_{\tilde D_n}(p,\partial \tilde D_n).\]
It follows that for any $p \in B_{\tilde D_n}(p_n, \frac{\tilde R_n}{2})$ we have
that $|A_n|(p) \le 2|A_n|(p_n)$.

Let $\lambda_n = |A_n|(p_n)$ and consider the sequence of rescaled surfaces
\[\lambda_n(\tilde D_n - p_n).\]
Note that $\lambda_n(\tilde D_n  - p_n)$ is an
$\frac{H_n}{\lambda_n}$-surface and the norm of its second fundamental is
bounded above by $2$ in $B_{\R^3}(0, \frac{\lambda_n \tilde R_n}{2})$
with the dilated pull-back metric. The distance between the origin 
$0 \in \lambda_n(\tilde D_n - p_n)$ and the 
boundary of $\lambda_n(\tilde D_n - p_n)$ is at least $\lambda_n \tilde R_n$.

The arguments in \cite{Rosenberg-General-Curvature-Estimates} show that there is
a complete immersed $H$-surface $S$ in the euclidean 3-space
$(\R^3, g_{eucl})$ such that:
\begin{enumerate}
\item $0 \in S$;
\item $|A_S|(0) = 1$; 
\item $\|A_S\|_{C^0} \le 5$;
\item passing to a subsequence, there are pieces 
$\Delta_n \subset \lambda_n(\tilde D_n - p_n)$ such that
\[ \Delta_n \to S\]
in the sense that there is a $\delta' >0$ such that for any $x \in S$, and $n$
large enough, the components of $\Delta_n$ close to $B_S(x,
\delta)$ are graphical over $D(\delta) \subset T_x S$ and converge to
$B_S(x,\delta)$ in $C^2$-norm.
\end{enumerate}

For $R'$ large enough, the index of
\[B_{\Sigma_n}(p_n,R'/\lambda_n)\]
is at least one for $n$ sufficiently large, else $S$ must be stable,
contradicting the fact that $S$ is
complete and non-flat \cite{Rosenberg-General-Curvature-Estimates}.
By induction there is a set 
\[\hat \B_n \subset \hat \Sigma_n = \Sigma_n \setminus 
B_{\Sigma_n}(p_n, R'/\lambda_n)\]
with at most $I-1$ points such that
\[|A_n|(x)\min\{d_{\hat \Sigma_n}(x, \partial \hat \Sigma_n \cup \hat
\B_n),(\sqrt \Lambda)^{-1}\} \le C\]
We claim that we can take $\B_n = \hat \B_n \cup \{p_n\}$. Assume there is $z_n
\in \Sigma_n$ such that
\[|A_n|(z_n)\min\{d_{\Sigma_n}(z_n, \partial\Sigma_n \cup  \B_n),
(\sqrt \Lambda)^{-1}\} \to \infty.\]

We will consider two cases. First assume
\[\liminf_{n \to \infty} \lambda_n d_{\Sigma_n}(z_n, p_n) < \infty.\] 
Passing to a subsequence we get that
\[\frac{d_{\Sigma_n}(z_n, p_n)}{\tilde R_n} = 
\frac{\lambda_n d_{\Sigma_n}(z_n,p_n)}{\lambda_n \tilde R_n} \to 0.\]
So, for $n$ large enough $z_n \in B_{\Sigma_n}(p_n, \tilde R_n/2)$, and hence
\[|A_n|(z_n) \le 2\lambda_n.\]
We conclude that
\[\lim_{n \to \infty} |A_n|(z_n)d_{\Sigma_n}(z_n, p_n) 
\le \lim_{n \to \infty} 2\lambda_n d_{\Sigma_n}(z_n, p_n) < \infty.\]
Contradicting the choice of $z_n$. Now assume that 
\[\liminf_{n \to \infty} \lambda_n d_{\Sigma_n}(z_n, p_n) = \infty.\]
It follows that for $n$ large enough $d_{\Sigma_n}(z_n, p_n) > R'/\lambda_n$,
hence $z_n \in \hat \Sigma_n$, contradicting the inductive hypothesis. \qed
\begin{remark}
Note that this same proof works in the case where $M$ has a boundary as long as
$\Sigma$ is away from this boundary.
\end{remark}

\section{Picture near blow-up points}
\label{sec: local picture near blow-up pts}

In this section we will consider sequences embedded CMC surfaces $\{\Sigma_n\}$
in a fixed closed 3-manifold $(M,g)$. Since the $\Sigma_n$ are embedded and the
ambient manifold is fixed we can make sense of a sequences of rescaled surfaces
\[\sigma_n(\Sigma_n - p_n)\]
in a slightly different way to that presented in Subsection \ref{subsec: harm
coords and sequences of}. Namely, we can take harmonic charts $(U_n, \varphi_n, B_M(x,p_n))$ and
define 
\[\Sigma_n - p_n \subset (U_n, \varphi_n^*g_n)\]
to be the pre-image of the component of $\Sigma_n \cap B_M(x,p_n)$ 
passing by $p_n$. Then
$\sigma_n(\Sigma_n - p_n)$ has as the properties listed in Subsection
\ref{subsec: harm coords and sequences of} in addition to being embedded.

Following \cite{Davi-Minimal-hypersurfaces-with-bdd-index} we say that a
sequence of finite sets of points $\B_n \subset \Sigma_n$ is a 
\textit{sequence of blow-up sets} if:

\begin{enumerate}
\item The curvature blows up at the $\B_n$, that is, taking
$\lambda_n(p) = |A_n|(p)$ we have
\[\liminf_{n \to \infty} \min_{p \in \B_n} \lambda_n(p) \to \infty.\]

\item Taking any sequence of points $p_n \in \B_n$ we can pass to a subsequence so
	\[\lambda_n(p_n)(\Sigma_n - p_n) \to \mathcal{L}_\infty,\]
where $\mathcal L_\infty$ is a weak CMC lamination of euclidean 3-space
with $\|A_{\mathcal L_\infty}\|_{C^0} \le 5$ and 
$|A_{\mathcal L_\infty}|(0) = 1$.

\item The blow-up points do not appear in the blow-up limit of the other points,
	that is
\[\liminf_{n \to \infty}\min_{p, q \in \B_n, p\neq q}
\lambda_n(p) d_{M}(p, q) = \infty\]
\end{enumerate}

\begin{proposition}\label{prop: blow-up set}
Suppose that the index of the $\Sigma_n$ is uniformly bounded above by a constant
$I < \infty$. Then there is a sequence of blow-up sets $\B_n \subset \Sigma_n$ with 
cardinality at most $I$ and a constant $C < \infty$ depending only on $I$
such that
\begin{equation}\label{eq: curvature bounds embedded case}
|A_n|(x)\min\{d_M(x, \partial \Sigma_n \cup \B_n), (\sqrt \Lambda)^{-1} \} \le C
\end{equation}
for all $n$ and all $x \in \Sigma_n$.
\end{proposition}

\begin{proof}
We will show that one can take the sets $\B_n$ as in the proof of Theorem
\ref{thm: curvature bounds} by
changing the distance function from the intrinsic
distance to the extrinsic distance. It is clear that these sets have cardinality at
most $I$ and that the $\lambda_n(p_n)(\Sigma_n - p_n)$
are embedded with bounded second fundamental form and the distance from the
origin to boundary of $\lambda_n(p_n)(\Sigma_n - p_n)$ diverges.
So it follows that they must subconverge to a weak CMC lamination of euclidean 3-space.
(see \cite{Colding-Minicozzi-The-Space-of-embeded-min-surf-of-fixed-genus-IV}*{Appendix
B})

To show the third item in the definition holds, let $\hat \B_n$ be as constructed
in the proof of Theorem \ref{thm: curvature bounds} and assume by contradiction
that 
\[\liminf_{n \to \infty}\min_{q \in \hat \B_n} \lambda_n d_{M}(p_n, q) <
\infty.\]
Take $q_n$ to be a point in $\hat \B_n$  closest to $p_n$ and let
$\eta_n = |A_n|(q_n)$, repeating the arguments from the proof of Theorem
\ref{thm: curvature bounds} we have that, passing to a
subsequence, for $n$ large enough
\[\eta_n \le \frac{\lambda_n}{2}.\]
So $ \liminf_{n \to \infty} \eta_nd_{M}(q_n, p_n) < \infty.$
This clearly implies 
\[\lim_{n \to \infty}|A_n|(q_n)d_{M}(q_n, \partial \hat \Sigma_n) <
\infty\]
contradicting the choice of $q_n$.
\end{proof}

In the case where the mean curvature 
of the $\Sigma_n$ are uniformly bounded, repeating the arguments
in \cite{Davi-Minimal-hypersurfaces-with-bdd-index}*{Section 4} we can conclude:

\begin{lemma}\label{lemma: picture near blow-up sets}
Suppose $\{\Sigma_n\}$ is a sequence of CMC surfaces embedded in a complete
Riemannian 3-manifold $(M,g)$ with mean curvature uniformly bounded above by 
a constant $H$. Suppose the blow-up sets $\B_n$
accumulate on a set $\B_\infty$. Then there is a $\delta > 0$
such that the set 
\[B_M(\B_\infty, \delta) = \bigcup_{p_\infty \in \B_\infty} B_M(p_\infty, \delta)\]
is a union of disjoint balls where the following holds: Write
$\Sigma'_n$ for the connected components of 
$B_M(\B_\infty, \delta) \cap \Sigma_n$ containing points
in $\B_n$ and $\Sigma''_n$ for the other connected components of
$B_M(\B_\infty, \delta) \cap \Sigma_n$.
Then there are functions $m = m(I,H)$ and $r = r(I,H)$ such that for $n$ large:
\begin{enumerate}
\item The surfaces $\Sigma_n''$ are CMC disks with second fundamental form
	uniformly bounded above.
\item The surfaces $\Sigma_n'$ intersects $\partial B_M(\B_\infty,\delta)$ transversely
in at most $m$ simple closed curves.
\item The surfaces $\Sigma_n'$ have genus at most $r$;
\item The surfaces $\Sigma_n'$ have uniformly bounded area.
\end{enumerate}
\end{lemma}

\section{Picture away from blow-up points}\label{sec: picture away from blow-up}

In this section we will show that a sequence of CMC surfaces embedded in a space
with finite fundamental group and with mean curvature bounded away from zero cannot
accumulate in a region where they have uniformly bounded second fundamental 
form (Lemma \ref{lemma: local picture away from blow-up}). The results
presented here will be used to prove Theorem \ref{thm: area bound} and will also
be needed to remove any upper mean curvature bound assumption on Theorem
\ref{thm: topology bound}. We will state and prove some of our results for hypersurfaces
in manifolds of arbitrary dimension, but for our applications we only need to
consider surfaces in 3-manifolds.

\subsection{Two-sided hypersurfaces in spaces with finite fundamental group}
\label{subsec: orentable surf in spaces}
Let $(M,g)$ be a Riemannian manifold (possibly with boundary) with
finite fundamental group. Let $\Sigma \subset M$ be a complete properly embedded
two-sided hypersurface and let $N$ be a unit normal vector field of $\Sigma$.
Suppose $\gamma: [0,1] \to M$  is a differentiable curve between points in
$M \setminus \Sigma$ which is transversal to $\Sigma$.
Write $\gamma^{-1}(\Sigma) = \{t_1, \cdots, t_n\}$ where $t_1 < \cdots < t_n$.
We say that $\Sigma$ is $\ell$-\textit{alternating} if for any such path 
$\gamma$ one of the following holds:
\begin{itemize}
	\item Either $\ell > n - 1$;
	\item or there is a $j \le \ell +1 $ such that
\[g(\gamma'(t_1), N(\gamma(t_1)))\ \text{and}\ 
g(\gamma'(t_{j}), N(\gamma(t_{j})))\]
have opposite signs.
\end{itemize}

\begin{proposition}\label{prop: alternating surfaces}
Suppose $|\pi_1(M)| = \ell < \infty$. Then any complete properly embedded two-sided 
hypersurface $\Sigma \subset M$ is $\ell$-alternating.
\end{proposition}
\begin{proof}

Let $\tilde M$ be the universal cover of $M$ and let $\tilde \Sigma$ be the lift
of $\Sigma$ in the sense that, for any evenly
covered open set $U \subset M$ and any sheet $\tilde U \subset \tilde M$
over $U$, $\tilde \Sigma \cap \tilde U$ is mapped isometrically to $\Sigma \cap
U$ by the covering map. Note that we can lift a unit normal vector field
$N$ of $\Sigma$ to a unit normal vector field $\tilde N$ of $\tilde \Sigma$.

Let $\gamma:[0,1] \to M$ be a differentiable path with ends away from $\Sigma$
which is transversal to $\Sigma$ and intersects $\Sigma$ it in at least $\ell +1$
points. Choose any lift $\tilde \gamma$ of $\gamma$. Note that
$\gamma^{-1}(\Sigma) = \tilde \gamma^{-1}(\tilde \Sigma)$ and for each $t_j$ 
in this set
\[g(\gamma'(t_j), N(\gamma(t_j))) =
\tilde g(\tilde \gamma'(t_j), \tilde N(\tilde \gamma(t_j))).\]
So we only need to show that there are $j, j' \le \ell +1$ such that
\[g(\tilde\gamma'(t_j), \tilde N(\gamma(t_j))) \text{ and }
\tilde g(\tilde \gamma'(t_{j'}), \tilde N(\tilde \gamma(t_{j'})))\]
have opposite sings.

Since the order of the cover $\tilde M \to M$ is $\ell$, $\tilde \Sigma$ has at most
$\ell$ connected components. So there must be a connected component $\tilde
\Sigma_0$ of $\tilde \Sigma$ intersecting $\tilde \gamma$ at least twice.
Let $\tilde \gamma(s_1) = \tilde x, \tilde \gamma(s_2) = \tilde y \in \tilde
\Sigma_0$ be the first and second points of intersection of  $\tilde \gamma$
with $\tilde \Sigma_0$ respectively. Assume by contradiction that 
the product of the velocity of $\tilde \gamma$ with the normal of
$\tilde \Sigma_0$ has the same sign at both these points. By changing the
direction of the normal of $\tilde
\Sigma_0$ we can assume that these are both negative. Choose a path $\eta: [0,1]
\to \tilde \Sigma_0$ between $\tilde x$ and $\tilde y$ and for $r \ge 0$ set
\[\eta_r(t) = \exp_{\tilde M}(r\tilde N(\eta(t))).\]
Fix $\rho > 0$ small enough so that $\eta_\rho$ does not intersect $\tilde
\Sigma_0$.
Taking $\epsilon >0$ small enough, we can join $\gamma|_{[s_1 - \epsilon,
s_2 - \epsilon]}$ to $\eta_\rho$ by two smooth curves that do not intersect
$\tilde \Sigma_0$. Let $\beta$ be the concatenation of these 4 curves. Then
$\beta$ is a smooth by parts and intersects $\tilde \Sigma_0$ exactly at 
$\tilde x$. On the other hand, since $\tilde M$ is simply-connected and 
$\tilde \Sigma_0$ is complete and proper, the intersection number of $\beta$ and $\Sigma_0$
has to be even.  
\end{proof}

\subsection{CMC surfaces in manifolds with finite fundamental group.}
We finish this section with the proof of the lemma below.
\begin{lemma}\label{lemma: local picture away from blow-up}
Let $(M,g)$ be a 3-manifold (possibly with boundary) with absolute sectional
curvature bounds $|K| \le \Lambda < \infty$ and let $\Sigma \subset M$ be 
a properly embedded CMC surface. Fix any $\alpha \in (0,1)$ and let $\Omega \subset M$
be an open set such that there are $\delta, i >0$ with the property that
$\Omega(\delta)$ as defined in Theorem \ref{thm: harmonic coords} does not
intersect $\partial M \cup \partial \Sigma$ and $\text{inj}_{(M,g)}(p) \ge i$
for all $p \in \Omega(\delta)$. Suppose the second fundamental form of $\Sigma$
in $\Omega$ is bounded above by a constant $C < \infty$. Then for all 
$p \in \Omega$ there is a
$\epsilon >0$ depending only on $\alpha, \delta, i, C$ and $\Lambda$ such 
that $\Sigma \cap B_M(p,\epsilon)$ is parametrized by
graphs of functions with $C^{2,\alpha}$ bounds depending only on 
$\alpha, \delta, i, C$ and $\Lambda$.

If we assume additionally that $\pi_1(M)$ is finite and
that $\Sigma$ is complete and has mean curvature mean
curvature $H \ge \eta > 0$, then there is a $N < \infty$ depending only on 
$\alpha, \delta, i, C, \Lambda, \eta$ and $|\pi_1(M)|$ such that 
$\Sigma \cap B_M(p,\epsilon)$ is parametrized by no more than $N$ graphs.
\end{lemma}
\begin{proof}
The first part of the statement is equivalent to the Uniform Graph Lemma
presented in Appendix \ref{app: uniform graph}. Now for the second part of the
statement, note that by Proposition \ref{prop: alternating surfaces} the graphs that
parametrize $\Sigma$ must alternate. Since these graphs must have alternating
mean curvature (in the sense of Appendix \ref{app: alternating CMC graphs}) it
is possible to bound the number $N$ of graphs uniformly using the mean curvature
equation, as shown in Proposition \ref{prop: finite multiplicity}.  
\end{proof}

\section{Proof of the main theorems}
\label{sec: pf of main theorems}

Now we are are ready to prove theorems \ref{thm: topology bound} and \ref{thm:
area bound}. Since the proofs are very similar we present the proof of Theorem
\ref{thm: topology bound} in more detail and reference it in the proof of
Theorem \ref{thm: area bound}.

\subsection{Proof of Theorem \ref{thm: topology bound}}
Let $(M,g)$ be a closed Riemannian 3-manifold and let $\{\Sigma_n\}$ be
sequence of closed embedded CMC surfaces in $M$ with index at most $I$
and area at most $A_0$. Suppose by contradiction that the genus of the 
$\Sigma_n$ form a divergent sequence. Let $H_n$ be the mean curvature of 
$\Sigma_n$. We will consider two cases depending on the behavior of the sequence
$\{H_n\}$.

\vspace{5pt}
\textbf{Case 1:} $\liminf_{n \to \infty} H_n < \infty$.\\
Pass to a subsequence so the $H_n$ are uniformly bounded in $n$ and divide the 
$\Sigma_n$ into three parts:
$\Sigma_n', \Sigma_n''$ as in Lemma \ref{lemma: picture near blow-up sets} and
\[\Sigma_n^b = \Sigma_n \setminus B_M(\B_\infty,\delta/2).\]

By Theorem \ref{thm: curvature bounds}, the norm of the second fundamental form of the 
$\Sigma_n$ is uniformly bounded away from $B_M(\B_\infty, \delta/3)$.
Since the area of the $\Sigma_n$ is uniformly bounded above the number of graphs in
the local parametrization of Lemma \ref{lemma: local picture away from
blow-up} is uniformly bounded. So
we can pass to a subsequence such that the $\Sigma_n^b$ converge
with finite multiplicity to a CMC surface $\Sigma_\infty^b$ (see 
\cite{Meeks-Ros-Rosemberg-Global-Theory-of-min-surf}*{Theorem 4.37}). Hence for
$n$ large enough $\Sigma_n^b$ is a cover of $\Sigma_\infty^b$ with uniformly
bounded degree, so there are uniform bounds on the Betti numbers of the 
$\Sigma_n^b$.

Lemma \ref{lemma: picture near blow-up sets} together with the area bound limits the
Betti numbers of $\Sigma_n'$ and $\Sigma_n''$. By
\cite{Davi-Minimal-hypersurfaces-with-bdd-index}*{Lemma 3.1}
the $\Sigma_n^b$ intersects $\Sigma_n'$ and $\Sigma_n''$ on annuli, so we can
use a Mayer-Vietoris sequence to bound on the Betti numbers of 
the $\Sigma_n$
\footnote{
We can also use a surgery argument
\cite{Davi-Minimal-hypersurfaces-with-bdd-index}*{Proposition 5.1}
}.
This contradicts the assumption that the sequence of the genus of the $\Sigma_n$ diverges.

\vspace{5pt}
\textbf{Case 2:} $H_n \to \infty$.\\
By Corollary \ref{cor: diam bound} we can conclude that a subsequence
of the $\Sigma_n$ converges to a point $p_\infty \in M$ in the Hausdorff
sense. So for $n$ large enough 
\[\Sigma_n \subset B_M(p_\infty, r).\]
where $r = inj_{(M,g)}(p_\infty)/2$.
Using the exponential map of $M$ at $p_\infty$ we can pull $\Sigma_n$
back to a properly embedded $H_n$-surface
\[\tilde \Sigma_n \subset (B_{\R^3}(p_\infty, r),
\exp_{p_\infty}^*g).\]
Scaling the $\tilde \Sigma_n$ by $H_n$ we obtain a sequence of $1$-surfaces
\[\overline \Sigma_n = H_n \tilde\Sigma_n\]
with uniformly bounded diameter in an increasing sequence of balls with metrics 
converging smoothly to the euclidean metric in compact sets.

Consider a decomposition of the $\overline \Sigma_n$ into 
$\overline\Sigma_n', \overline\Sigma_n''$ and $\overline\Sigma_n^b$ analogous to the 
decomposition of the $\Sigma_n$ in Case 1. By Lemma \ref{lemma: local picture away
from blow-up} there are uniform bounds on the number of sheets of the
$\overline\Sigma_n^b$ on any small open set. Hence we can conclude that, passing
to a subsequence, the $\overline\Sigma_n^b$ converge to a 1-surface
$\Sigma_\infty^b$. These same arguments can be used to bound the number of
connected components of the $\overline \Sigma_n''$. Finally we can conclude 
using the same arguments from Case 1. \qed

\subsection{Proof of Theorem \ref{thm: area bound}}\label{subsec: pf of thm 1.2}
First note that by Theorem \ref{thm: topology bound} we
only have to show that there is an uniform area bound.
Let $(M,g)$ be a spherical Riemannian 3-manifold and let $\{\Sigma_n\}$ be a
sequence of closed embedded CMC surfaces in $M$ with index at most $I<\infty$
and mean curvature at least $\eta > 0$. Assume by contradiction that 
the area of the $\Sigma_n$ form a divergent sequence.

Let $H_n$ be the mean curvature of $\Sigma_n$. In the case where
$\limsup_n H_n \to \infty$ we can pass to a subsequence so $H_n \to \infty$.
Following the arguments used in Case 2 of the proof of Theorem 
\ref{thm: topology bound} we can conclude that the area of the mean-curvature
rescaled $\Sigma_n$ is uniformly bounded, hence the area of the $\Sigma_n$ must
converge to 0, a contradiction.

So we may assume $\limsup_{n \to \infty} H_n < \infty$. In this case we can apply the
same decomposition used in Case 1 on the proof of Theorem \ref{thm: topology
bound} and use Lemma \ref{lemma: local picture away from blow-up} to bound the
number of sheets of any $\Sigma_n^b$ that can pass through small regions. A
similar argument can be used to bound the number of connected components the 
$\Sigma_n''$ can have, and hence we can conclude that there are uniform area bounds for
these two parts. Finally, Lemma \ref{lemma: picture near blow-up sets} allows us
to conclude that there are uniform area bounds for the $\Sigma_n'$. \qed
\appendix

\section{Uniform Graph Lemma}\label{app: uniform graph}

The following result is a version of the Uniform Graph Lemma (cf
\cite{Meeks-Ros-Rosemberg-Global-Theory-of-min-surf}*{Lemma
4.35}) which allows for $C^{2,\alpha}$ bounds for functions parametrizing CMC
surfaces depending only on the second fundamental form of the surface and
the ambient space sectional curvature. This lemma is slightly stronger than what we
need in order to show our main theorems, however it is included here for the
sake of completion.

\begin{lemma}\label{lemma: Uniform graph}
Suppose $(M,g)$ is a 3-manifold with absolute sectional curvature bounds 
$|K| \le \Lambda < \infty$ and let $\Sigma \subset M$ be a properly embedded 
CMC surface. Let $\Omega \subset M$ be an open set lying away from the boundary
of $M$ and suppose the second fundamental form of $\Sigma$ in $\Omega$ is
bounded above by a constant $C < \infty$.
Fix any $\alpha \in (0,1)$ and
suppose $\delta, i,r_0$ and $Q_0$ are as in Theorem \ref{thm: harmonic coords}.
Fix an $r \in (0, r_0)$ and let $x \in \Omega$ be such that 
$d_M(x, \partial \Sigma) > r.$

Choose a harmonic chart $(U, \varphi, B_M(x,r))$ as in
Theorem \ref{thm: harmonic coords}.
For any $\epsilon \in (0,r)$ let $\Sigma'$ be the set of connected
components of $\varphi^{-1}(\Sigma \cap B_M(x, r))$ intersecting 
the euclidean ball $B_{\R^3}(0,\epsilon)$. Then there are
$\epsilon> 0 $, $\rho \in (\epsilon, r)$ and $C' < \infty$ depending only on
$\Lambda, C, i$ and $\alpha$, and a rotation $R \in O(\R^3)$
such that:
\begin{enumerate}
\item Every connected component of $R(\Sigma') \cap B(\rho) \times \R$
is the graph of a function $u$ over $B(\rho)$.
\item  For all such functions $u$ we have 
\[\|u\|_{C^{2,\alpha} (B(\rho))} \le C'.\]
\end{enumerate}
\end{lemma}
\begin{proof}
From \eqref{eq: quadatic form bound}
and \eqref{eq: holder bound on tensor}
we can conclude that there is a $K < \infty$ depending only on $Q_0$ and $C$ 
such that the norm of the euclidean second fundamental form of 
$\varphi^{-1}(\Sigma \cap B_M(x, r))$ is at most $K$. For any 
$\epsilon >0$ let $\Sigma'$ be as above and let
$G': \Sigma' \to \mathbb{RP}^2$ be the euclidean unoriented Gauss map of $\Sigma'$.

We first claim that there is an $\epsilon >0$ depending only on $K$ and $r$ such
that
\[\sup_{p, q \in \Sigma'\cap B_{\R^3}(0, \epsilon)}
d_{\mathbb{P}^2}(G'(p),G'(q)) < \frac{\pi}{4}.\]
For $p$ and $q$ in the same connected component, the existence of $\epsilon$ 
follows immediately from the bounds on the second fundamental form. 
For $p, q$ in distinct connected components 
we can use the bound on the second form to show that if $\epsilon$ is
sufficiently small and $G'(p)$ and $G'(q)$ 
are too far apart then the connect components on $\Sigma'$ containing $p$ and $q$
must cross.

Choose a rotation $R$ so that the tangent planes to 
$R(\Sigma') \cap B_{\R^3}(0,\epsilon)$ are bounded away from any vertical plane 
by at least $\pi/6$. It follows that there is a constant $\tilde \rho$
depending only on $K$ such
that, for all $p = (p_1,p_2,p_3) \in R(\Sigma) \cap B_{\R^3}(0,\epsilon)$ 
the portion of $R(\Sigma')$ in the cylinder $B((p_1,p_2),\tilde \rho) \times \R$ is 
the graph of a function over $B(p,\tilde \rho)$.
The bounds on the euclidean second form and on the distance from the tangent
planes to vertical planes give bounds on the $C^2$ norms for 
these functions.

By possibly making $\epsilon$ smaller, we can assume that $\epsilon < \tilde
\rho /2$. So taking $\rho = \tilde \rho/2$, we have that every connected
component of $R(\Sigma') \cap B(\rho) \times \R$ is a graph over $B(\rho)$ of a
function with $C^2$-control. Making $\rho$ smaller and
applying interior Schauder estimates 
\cite{Gilbarg-Trudinger-Elliptic-PDES}*{Theorem 6.2} using the mean curvature
equation we can obtain $C^{2,\alpha}$ bounds for these functions.  
\end{proof}

\section{Alternating CMC graphs}\label{app: alternating CMC graphs}

Here we will use a local version of the alternating property defined in 
\ref{subsec: orentable surf in spaces} to bound the number of sheets any CMC 
hypersurface can have in a small region.

Let $B(r) \subset \R^d$ be an open ball of radius $r$ and suppose $g$ is a 
metric in $B(r) \times \R$ such that the $g_{ij}$ have
$C^{1,\alpha}$-norm bounded by a constant $Q_0$ for some fixed
$\alpha \in (0,1)$. Assume further that, as quadratic forms
\[Q_0^{-1} \delta_{ij} \le g_{ij} \le Q_0 \delta_{ij}.\]
For any $u \in C^2(B(r))$ we will choose the normal $N_u$ to the
graph of $u$ pointing upwards, that is $g(N_u, \partial_d) > 0$. For 
$x \in B(r)$ let $h_u(x)$ 
be the scalar mean curvature of the graph of $u$ at $(x,u(x))$ with respect to
the metric $g$ and the normal $N_u$.
We will refer to $h_u$ as the mean curvature of $u$, and if $h_u$ is constant we
will say that $u$ has constant mean curvature.

\begin{proposition}\label{prop: unif L-infy bound}
Suppose $u, v \in C^{2, \alpha}(B(r))$ have constant mean curvature with 
opposite signs and
\begin{enumerate}
\item $u \le v$ point-wise;
\item  $|h_u| = |h_v| \ge \eta$ for some $\eta > 0$;
\item $\|u\|_{C^{2,\alpha}(B(r))}, \|v\|_{C^{2,\alpha}(B(r))} \le C$ for some $C > 0$.
\end{enumerate}
Then there is a $c > 0$ depending only on $Q_0, \alpha, r, \eta$ and $C$ such that 
\[\|u-v\|_{C^0(B(r))} > c\]
\end{proposition}
\begin{proof}
Since $u$ and $v$ have $C^{2,\alpha}$ norm bounded above by $C$ we can
conclude that for all $\epsilon >0 $ there is a $\delta >0$ depending only on 
$C, \alpha, \epsilon$ and $r$ such that
\[\|u-v\|_{C^2(B(r/2))} < \epsilon\]
whenever $\|u-v\|_{C^0(B(r))} < \delta$.
So if $u$ and $v$ can be made arbitrarily close in $B(r)$, their mean curvatures
can also be made to be arbitrarily close in $B(r/2)$.

On the other hand from the mean curvature equation 
(see \citelist{
\cite{Colding-Minicozzi-Minimal-surfaces}*{Chapter 5}
\cite{Colding-Minicozzi-A-course-in-min-surf}*{Chapter 7}})
we can conclude that there is a $C'< \infty$ depending only on $Q_0, \alpha, r$
and $C$ such that
\[2\eta \le \|h_u - h_v\|_{C^0(B(r/2))} \le C'\|u - v\|_{C^2(B(r/2))}\]
whenever $u$ and $v$ are $C^2$-close enough. Contradicting the fact that $u$ and
$v$ can be made arbitrarily close in the $C^2$ norm.  
\end{proof}

Assume $u_1, \cdots, u_n \in C^{2,\alpha}(B(r))$ have non-zero constant mean
curvature with the same magnitude and $u_1 \le u_2 \le \cdots \le u_n$ 
point-wise. We say that such a collection is $\ell$-\textit{alternating} for 
some $\ell \in \mathbb N$ if for 
all $j \le n - \ell$ there is a $j' \le j + \ell$ such that the mean curvature
of $u_j$ and $u_{j'}$ have opposite signs. Note that by Proposition \ref{prop:
alternating surfaces} this property always holds for graphs parametrizing a 
complete properly embedded two-sided CMC hypersurface in a manifold with fundamental group of order at most
$\ell$.

\begin{proposition}\label{prop: finite multiplicity}
Let $u_1, \cdots, u_n \in C^{2,\alpha}(B(r))$ be an $\ell$-alternating
collection and let $\eta > 0, C < \infty$ and $\epsilon \in (0,r)$ be such that 
for all $k= 1, \cdots ,n$:
\begin{enumerate}
\item $|h_{u_k}| \ge \eta$;
\item $\|u_k\|_{C^{2,\alpha}(B(r))} \le C$;
\item the graph of every $u_k$ intersects an euclidean ball
	$B(\epsilon)$.
\end{enumerate}
Then there is an $N < \infty$ depending only on $Q_0,
\alpha, \eta, r, \ell, C$ and $\epsilon$ such that $n \le N$.
\end{proposition}
\begin{proof}
By passing to a subset of $\{u_1, \cdots, u_n\}$ with at least 
$\lfloor n/\ell\rfloor$ elements
we can assume that the mean curvature of the $u_k$ alternate 
(i.e. $\ell = 1$). The difference $w_{k, k+2} = u_{k+2} - u_k$ follows an
elliptic equation with ellipticity and coefficients controlled by the $C^1$ norm of
the metric and the $C^1$ norm of $u_{k+2}$ and $u_k$
(see \citelist{\cite{Colding-Minicozzi-Minimal-surfaces}*{Chapter 5}
\cite{Colding-Minicozzi-A-course-in-min-surf}*{Chapter 7}}).
It follows that these functions obey a Harnack inequality with a constant
$C'$ that only depends on $Q_0$ and $C$
\cite{Gilbarg-Trudinger-Elliptic-PDES}*{Theorem 8.20}.

It follows from the $C^1$ bound on the $u_k$ that 
$|u_k| < 3C\epsilon$ in $B(\epsilon)$. Hence for at least one of the 
$w_{k, k+2}$ must have
\[\inf_{x \in B(\epsilon)}w_{k, k+2}(x) \le \frac{6C\epsilon}{n}.\]
So by the Harnack inequality $w_{k,k+2} \le 6C' \frac{C\epsilon}{n}$ in
$B(\epsilon/2)$, this implies that
\[\|u_k - u_{k+1}\|_{C^0(B(\epsilon/2))} \le 6C'\frac{C\epsilon}{n}.\]
So we must have $n < 6 \epsilon C' C c^{-1}$ where $c$ is as in Proposition 
\ref{prop: unif L-infy bound} with $r = \epsilon/2$.  

\end{proof}
\bibliography{refs}

\end{document}